\documentclass[12pt]{article}

\usepackage{latexsym,amssymb,amsmath}
\newcommand{\nn}{\lhd}

\newcommand{\z}{\mathbb Z}

\newtheorem{lem}{Lemma}[section]

\newtheorem{co}[lem]{Corollary}
\newtheorem{thm}[lem]{Theorem}
\newtheorem{prop}[lem]{Proposition}

\newtheorem{qu}[lem]{Question}
\newenvironment{proof}{\textbf{Proof.}}{\newline\hspace*{\fill}{$\Box$}\\}

\begin{document}
\title{A formula for the normal subgroup growth of Baumslag-Solitar groups}
\author{J.\,O.\,Button\\
Selwyn College\\
University of Cambridge\\
Cambridge CB3 9DQ\\
U.K.\\
\texttt{jb128@dpmms.cam.ac.uk}}
\date{}
\maketitle
\begin{abstract}
We give an exact formula for the number of normal subgroups of each finite
index in the Baumslag-Solitar group $BS(p,q)$ when $p$ and $q$ are coprime.
Unlike the formula for all finite index subgroups, this one distinguishes
different Baumslag-Solitar groups and is not multiplicative. This allows
us to give an example of a finitely generated profinite group which is
not virtually pronilpotent but whose zeta function has an Euler product. 
\end{abstract}
\section{Introduction}

A finitely generated group $G$ has finitely many subgroups $a_n(G)$ of each
finite index $n$ and hence finitely many normal subgroups $a_n^\nn(G)$.
If $R$ is the finite residual of $G$, that is the intersection
of all finite index subgroups, then $a_n(G)=a_n(G/R)$ and $a_n^\nn(G)=
a_n^\nn(G/R)$ by the correspondence theorem, so it is enough to consider
only residually finite groups.
There has been much recent activity in this area on the behaviour and
asymptotic growth of these functions. There has also been work on
obtaining exact formulae for $a_n(G)$ when $G$ lies in particular classes
of groups. However much less is known about exact formulae for 
$a_n^\nn(G)$. The only cases we can find in the literature are\\
(1) Some finite groups.\\
(2) Some abelian groups (where $a_n(G)=a_n^\nn(G)$ anyway).\\
(3) Some torsion-free nilpotent groups of class 2 in \cite{frzds}. Actually
the normal zeta functions $\zeta_G^\nn(s)$ are given from which we can
deduce $a_n^\nn(G)$.\\
(4) The 17 wallpaper groups have $a_n^\nn(G)$ (or rather $\zeta_G^\nn(s)$)
listed in \cite{dsms}. These groups are  all finite extensions of the
free abelian group of rank 2.\\
(5) Groups $G$ where $G/R$ is in (1) to (4).\\

In this note we give a formula for the famous class of 2 generator 1 relator
groups called the Baumslag-Solitar groups. These were introduced in \cite{bs}
and the group $BS(p,q)$ has presentation
\[\langle t,a|ta^pt^{-1}=a^q\rangle
\mbox{ for }p,q\in\z-\{0\}.\]
We consider the case when $p$ and $q$ are coprime and prove that
\begin{equation}a_n^\nn(BS(p,q))=\sum_{\substack{
d|n\\d|q^{(n/d)}-p^{(n/d)}}}
\mbox{gcd}(d,q-p).
\end{equation}
In \cite{gel} similar looking formulae in the case of all finite index 
subgroups were established for the same groups:
\begin{equation}
a_n(BS(p,q))=\sum_{\substack{
d|n\\ \mbox{gcd}(d,pq)=1}}d. 
\end{equation}
(In fact this formula in the case $p=1$ was earlier given in \cite{mols}
Theorem 2.)
Thus it is clear from (2) that $a_n(BS(p,q))$ is indistinguishable from
$a_n(BS(p',q'))$ when $pq$ and $p'q'$ have the same prime factors but we
show that the sequence $a_n^\nn(BS(p,q))$ uniquely determines the
Baumslag-Solitar group (and indeed $p$ and $q$ up to the obvious changes
of swapping $p$ and $q$ and replacing them both with their negatives).
We also prove that the sequence $a_n^\nn (BS(p,q))$ is not multiplicative,
unlike $a_n(BS(p,q))$. In \cite{dsms} it is asked whether a profinite
group whose zeta function has an Euler product is virtually pronilpotent.
We show that this is not the case for the profinite completion of
$BS(1,q)$ when $q\neq\pm 1$.
Finally we work out which finitely generated groups $G$ have $a_n(G)=
a_n^\nn(G)$ for all $n$.

\section{The formula}

Let $G$ be the Baumslag-Solitar group $BS(p,q)$ where $p$ and $q$ are
coprime. The key point which allows us to enumerate the finite images
of $G$, and hence the finite index normal subgroups, is the following:
\begin{prop}
If $G=BS(p,q)$ for $p$ and $q$ coprime then the quotient of $G$ by its
finite residual is metabelian.
\end{prop}
\begin{proof}
This is probably known but an interesting way to see this is to recall
that the proof of Malce'ev's result that a finitely generated residually
finite group is Hopfian actually shows that any element in ker$(\theta)$
for $\theta$ a surjective homomorphism of a finitely generated group $G$
is also in the finite residual $R$. Therefore on taking the famous
homomorphism $\theta(t)=t,\theta(a)=a^p$ which is onto if gcd$(p,q)=1$,
we have that for any $j\geq 1$ the commutator $[t^jat^{-j},a]$ is in
ker$(\theta^j)$ and so the normal closure of $a$ in $G/R$, which is
generated by the elements of the form $t^kat^{-k}$ for $k\in\z$, is
abelian.
\end{proof}
Of course if $p=1$ then $G$ is metabelian anyway, and residually finite
by a result of Philip Hall but otherwise we have a non-ascending HNN
extension of $\z$ so $G$ will contain a non-abelian free group.

This gives us
\begin{prop}
If $F$ is a finite image of $G=BS(p,q)$ for $p$ and $q$ coprime then $F$
is metacyclic with the order of $a$ in $F$ being coprime to $p$ and to $q$.
\end{prop}
\begin{proof}
Suppose we have $\theta:G\rightarrow F$. Then $H=\theta(\langle a^p\rangle)$
is conjugate in $F$ to $\theta(\langle a^q\rangle)$. If $H$ is strictly
contained in $\theta(\langle a\rangle)$ then by \cite{lub} we have that
$G$ is large. This happens if and only if $G/R$ is large but from above it is
metabelian. (Another way of seeing that $G$ is not large is to use \cite{epr}
Example 3.2 which shows that no finite index subgroup of $G$ has a homomorphism
onto $\z\times\z$.)

Consequently we have 
$\theta(\langle a^p\rangle)=\theta(\langle a^q\rangle)=
\theta(\langle a\rangle)=H$ and
if the order of $x=\theta(a)$ is $d$ then $p$ and $q$ are both coprime to
$d$. As $x^p$ and $x^q$ are generators of $H$ and $yx^py^{-1}=x^q$ where
$y=\theta(t)$, we see that the cyclic group $H$ is normal in $F$ and
$F/H=\langle yH\rangle$ is also cyclic.
\end{proof}
We can now deduce our main result.
\begin{thm}
If $p$ and $q$ are coprime then
\[a_n^\nn(BS(p,q))=\sum_{\substack{
d|n\\d|q^{(n/d)}-p^{(n/d)}}}
\mbox{gcd}(d,q-p).
\]
\end{thm}
\begin{proof}
We count normal subgroups $N$ of index $n$ in $G=BS(p,q)$ where $x$ has exact
order $d$ in the quotient. Fix $n$ and a factor $d$ and set $n=cd$. First
let us consider the case where $p=1$. We have that the group with presentation
\[\langle x,y|x^d,yxy^{-1}=x^q,y^c=x^s\rangle\]
for some $s$ defined modulo $d$ surjects to $F=G/N$ as these relations
all hold in $F$, where the last comes from the fact that $F/H$ is generated
by $yH$ and has order $c$. As $y^cxy^{-c}=x$ we must require of $c$ that $d$
divides $q^c-1$ or else $x$ would not have order $d$. Also 
$yx^sy^{-1}=x^s$ implies that $d$ divides $s(q-1)$ for the same reason.
Now with these conditions on $c$ and $s$ we have that the above presentation
is a metacyclic group of order $cd$ by \cite{joh} Chapter 3 Theorem 1, thus
this group is equal to $F$. Clearly $s$ is only defined modulo $d$ so the
number of allowable choices for $s$ is the number of integers $s$ from 1 to
$d$ such that $s(q-1)$ is 0 modulo $d$ which is gcd$(d,q-1)$. Moreover
different choices $s,s'$ modulo $d$ give rise to different normal subgroups
$N$ of $G$, or else we would have $y^c=x^s=x^{s'}$ in $G/N$ which again
contradicts the order of $x$. 

Now take a general $p$ and again fix $n$ and $d$. From Proposition 2.2 we
have that $d$ and $p$ are coprime so that in any finite image $F$ of
order $n$ in which $x$ has order $d$
we can replace $x$ with $u=x^p$, which also generates the
normal subgroup $H$ of $F$, to obtain the relation $yuy^{-1}=u^r$ where
$q\equiv pr$ modulo $d$. We can now follow the argument just as before to
obtain the claimed formula with $p$ replaced by 1 and $q$ by $r$.
However gcd$(d,q-p)=\mbox{gcd}(d,r-1)$ and $d$ divides $q^c-p^c$ if and
only if it divides $r^c-1$, so we are done.
\end{proof}

\section{Consequences}

It is pointed out in \cite{gel} (and in \cite{mols} for $p=1$) 
that the formula (2) for all finite
index subgroups only depends on the prime factors of $pq$ and so, as the
abelianisation $G/G'$ of a finitely generated group is the same
as that for $G/R$, we can create examples of infinitely many finitely generated
(and even residually finite)
groups $BS(p,q)$ (respectively $BS(p,q)/R$)
which are not isomorphic but which are isospectral, that
is the sequences $a_n(G)$ are the same. However in the case of normal
subgroups the sequence $a_n^\nn (BS(p,q))$ determines $p$ and $q$.
\begin{co}
If $p$ and $q$ are coprime with $q\geq |p|$ then the sequence\newline
$a_n^\nn(BS(p,q))$ uniquely determines $p$ and $q$. 
\end{co}
\begin{proof}
If $a_n^\nn (BS(p,q))=a_n^\nn (BS(p',q'))$ with $q-p>q'-p'\geq 0$ then
$a_{q-p}^\nn (BS(p,q))=\sigma(q-p)$ for $\sigma$ the sum of divisors
function. As we always have $a_n^\nn (BS(p,q))\leq\sigma (n)$ and strict
inequality if $|q-p|<n$, we see that $a_{q-p}^\nn (BS(p',q'))<\sigma (q-p)$.
Consequently we assume $q-p=q'-p'\geq 1$ as $q=p$ implies $p,p',q,q'$ are all
1. Let us take an odd prime $l$ with $l^m$ dividing $q+p$ but not $q'+p'$.
Then evaluating $a_{2l^m}^\nn (BS(p,q))$ involves summing over the odd
divisors $l^i$ and the even divisors $2l^i$ of $2l^m$. In the odd case we
obtain non-zero terms each time $l^i$ divides $q^{2l^{m-i}}-p^{2l^{m-i}}$,
which is always because $q+p$ does too, and these terms are always 1
because $p$ and $q$ are coprime. In the even case we pick up contributions
when $2l^i$ divides $q^{l^{m-i}}-p^{l^{m-i}}$. However if $l$ divides this
then it divides $q-p$ so the only non-zero term we have here comes from 2
if $q-p$ is even. As for $a_{2l^m}^\nn (BS(p',q'))$, first note that if we
pick up a term gcd$(d,q'-p')$ then it has the same value as for $BS(p,q)$
so we just need to show that we end up with less non-zero terms. But we have
the same contribution from the even divisors and we do not get $l^m$
dividing ${q'}^2-{p'}^2$ so this term of the sequence is strictly less.

This covers all cases except when $q+p=2^su>q'+p'=2^tu$ for $u$ odd where
we must have $t\geq 2$ and $q-p$ is 2 modulo 4. We similarly consider
the $2^{s+2}$-th term of each sequence and find that we pick up a term
from $2^{s+1}$ in the first case but not in the second.
\end{proof}
Note: It was shown in \cite{mold} that $BS(p,q)$ is isomorphic to 
$BS(p',q')$ if and only if $(p',q')=(p,q)$,$(-p,-q)$,$(q,p)$ or $(-q,-p)$.
Whilst Corollary 3.1 could be regarded as providing an alternative proof 
(at least when $p$ and $q$ are coprime), surely the quickest way to see 
this is to note that
the Alexander polynomial $\Delta(t)$ of $BS(p,q)$ is $pt-q$ and this is
defined up to multiplication by $\pm t^m$ for $m\in\z-\{0\}$. What Corollary
3.1 does establish is that for distinct Baumslag-Solitar groups $G$
with $p$ and $q$ coprime, we have that $G/R$ is distinct as well as the
profinite completion.

Two finitely generated groups $G_1,G_2$ will be isospectral if their profinite
completions are isomorphic, in which case we would have 
$a_n^\nn (G_1)=a_n^\nn (G_2)$ as well as the abelianisations $G_1/G_1'$ and
$G_2/G_2'$ being equal because the finite abelian images have to be the
same. The only isospectral groups known without
isomorphic profinite completions seem to be the fundamental groups of
the orientable and non-orientable surfaces of genus $g$, the example of
$\z\times\z\times C_2$ and $(\z\times\z)\rtimes C_2$ given in \cite{dsms}
after Theorem 1.4, and the Baumslag-Solitar groups in \cite{gel}. 
However none of these
examples provide isospectral pairs when we consider finite index normal
subgroups, because of Corollary 3.1 for Baumslag-Solitar groups
and by abelianising in the other cases.
Thus we ask the following:
\begin{qu} Are there examples of finitely generated groups $G_1$ and $G_2$
with non-isomorphic profinite completions but with both $a_n(G_1)=a_n(G_2)$
and $a_n^\nn (G_1)=a_n^\nn (G_2)$ for all $n$?
\end{qu}

Another property that we look for in the sequences $a_n(G)$ and $a_n^\nn (G)$
is that they are multiplicative, namely $a_{mn}^{(\nn)}
=a_m^{(\nn)}a_n^{(\nn)}$ when gcd$(m,n)=1$, as
if so then we obtain an Euler product for the zeta function or normal zeta
function of $G$. In \cite{puc} Theorem 1 it is shown that a finitely
generated profinite group $G$ is pronilpotent if and only if $a_n^\nn (G)$
is multiplicative. Consequently we can also answer the question after
Proposition 1.5 in \cite{dsms}
which asks that if the zeta function of a profinite group
$G$ has an Euler product then is $G$ virtually pronilpotent.
\begin{thm}
The profinite completion of $BS(1,q)$ for $q\neq\pm 1$ has the property that
its zeta function possesses an Euler product but it is not virtually
pronilpotent.
\end{thm}
\begin{proof}
We first show that for $p,q$ coprime we have $a_n(BS(p,q))$ is multiplicative
but $a_n^\nn(BS(p,q))$ is not (excepting of course $BS(1,1)=\z\times\z$).
On taking the formula (2) for $a_n(BS(p,q))$, and ignoring the gcd$(d,pq)=1$
condition in the sum, we have multiplicity of the sum of divisors function
$\sigma(n)$ which is proved by multiplying the sums $\sigma(m)$ and
$\sigma(n)$ together to obtain a new sum whose terms are exactly the
divisors of $mn$. Now put the condition back in and mark the terms in
$\sigma(m)$ and in $\sigma(n)$ that fail to be coprime to $pq$. When we
multiply out to get $\sigma(mn)$, the terms in this sum with one or both
factors marked are exactly the ones which are not coprime to $pq$. Hence by
erasing the marked terms on both sides we have multiplicity and thus we obtain
an Euler product for the zeta function of $BS(p,q)$ whenever $p$ and $q$ are
coprime, and hence for its profinite completion.

As for $a_n^\nn (BS(p,q))$, let us take an odd prime $l$ dividing $p^2+q^2$,
which will exist except when $(p,q)=(1,\pm 1)$ because $p^2+q^2$ is not
0 modulo 4. Now $a_l^\nn (BS(p,q))=1$ because $p^2+q^2$ and $q-p$ are
coprime (apart from a possible factor of 2) and $a_4^\nn (BS(p,q))$ is 1
if $q-p$ is odd, 3 if $q-p$ is 2 modulo 4 and 7 if $q-p$ is 0 modulo 4. But
as $l$ divides $q^4-p^4$ we have $a_{4l}^\nn (BS(p,q))$ is at least one more
than these numbers. (For $BS(1,-1)$ we have $a_2^\nn =3,a_3^\nn =1$ and
$a_6^\nn=4$.)

Thus by Puchta's result in \cite{puc} we have that the profinite completion
of $BS(p,q)$ is not pronilpotent unless $p=q=1$. Moreover on taking $p=1$
we have by \cite{mols} that the finite index subgroups of $BS(1,q)$ are
isomorphic to $BS(1,q^n)$ for $n\geq 1$, so if $q\neq\pm 1$ then the
profinite completion of $BS(1,q^n)$ is not virtually pronilpotent.
\end{proof}

Alternatively we can show that the profinite completion of $BS(p,q)$ is not
pronilpotent by directly finding a finite image which is not nilpotent.
Although if $p\neq 1$ we cannot then use the finite index subgroup trick as
above, an adaptation shows that the profinite completion is not virtually
pronilpotent in all cases other than $BS(1,\pm 1)$, although if $p\neq 1$ we
do not have an embedding of $BS(p,q)$ into its profinite completion. We
outline the proof: if $H$ has finite index $k_0$ in $G=BS(p,q)$ for $q>|p|$
and all the finite images of $H$ are nilpotent then any finite image of $G$
with order more than $k_0$ must have a non-trivial element whose centraliser
in $G$ has index at most $k_0$. But if we pick a prime $l>\mbox{max}(q^{k_0}
-p^{k_0},k_0,q)$ and let $k>0$ be the first value where $l$ divides $q^k-p^k$
then $k>k_0$ and $G$ has the finite image
\[F=\langle x,y|yx^py^{-1}=x^q,x^l,y^k\rangle =C_l\rtimes C_k.\]
But it is easily checked that $l$ being prime means that no non-trivial
power of $x$ commutes with no non-trivial power of $y$ and this in turn
implies that the centraliser of $x$ has index $k$ and all other centralisers
have index $l$.

Although we have assumed throughout, as in \cite{gel}, that $p$ and $q$
are coprime, it is reasonable to ask about the case of common factors.
However any results will be vastly different because now
$BS(p,q)$ is large by \cite{epr} Theorem 6, as if $g=\gcd (p,q)$ then we
have a surjection to the virtually free group $\z*C_g$ on adding the
relator $a^g$. This implies that it has
subgroup growth of strict type $n^n$ and, by \cite{mup} Theorem 1, normal
subgroup growth of strict type $n^{\log n}$. Both of these are the
fastest possible types for a finitely generated group. Moreover although
$BS(p,q)$ may still not be residually finite if common factors are present
(as this happens if and only if $|p|=|q|$), we have by \cite{ls} Theorem 3.1
that every finite group is an upper section. Hence a large group $L$ (and
any finite index subgroup of $L$ as this is also large) 
always has a finite non-soluble image and so the profinite completion of
a large group is never prosoluble.  Thus $a_n^\nn (BS(p,q))$ is not
multiplicative in this case by \cite{puc} and we can also show here that
neither is $a_n(BS(p,q))$.
\begin{prop}
If $G$ is a finitely generated group with $\lambda,\mu >0$ such that
$a_n(G)\geq\lambda e^{\mu n}$ for all large $n$ then $a_n(G)$ cannot be
multiplicative.
\end{prop}
\begin{proof}
If so then we would have $a_{n(n+1)}\geq\lambda e^{\mu n(n+1)}$ but as
$a_n\leq kn^{nd}$ if $G$ is a $d$-generator group (as this is also
true for the free group of rank $d$), we obtain
\[a_{n(n+1)}/a_{n(n+1)}\geq \lambda e^{\mu n(n+1)}/k^2n^{nd}(n+1)^{(n+1)d}\]
which tends to infinity as $n$ does.
\end{proof}
\begin{co}
If $p$ and $q$ are not coprime then $a_n(BS(p,q))$ is not multiplicative.
\end{co}
\begin{proof}
We have $a_n(BS(p,q))\geq a_n(\z*C_g)$ and without loss of generality we may
assume that $g$ is a prime, in which case we have the lower bound 
$\lambda ne^{(1-1/g)n(\log n-1)+n^{1/g}}$ in \cite{new}.
\end{proof}

We finish by giving the finitely generated groups such that $a_n(G)=
a_n^\nn (G)$. Here we use the fact due to Dedekind and Baer that a
non-abelian finite group $F$ has all its subgroups normal if and only if
$F=Q\times A\times B$ where $Q$ is the quaternion group of order 8, $A$ is
an elementary abelian 2-group and $B$ is an abelian group of odd order.
\begin{prop}
If $G$ is a finitely generated group where every finite index subgroup is
normal then either $G/R$ is abelian and equal to $G/G'$ or $G/R$ is finite
and equal to $Q\times A\times B$ as above.
\end{prop}
\begin{proof}
We have $R\leq G'$ so let us assume throughout that $R<G'$. Then $G$ has a
finite index subgroup $N$ with $G/N$ non-abelian. Now if $G/G'$ is infinite
then we have $M$ with $G/M=C_4$ and so $G/(M\cap N)$ is a finite non-abelian
group which surjects onto $C_4$, but as $Q/Q'=C_2\times C_2$ we have 
$G/(M\cap N)$ has a subgroup which is not normal and so $G$ has such a finite
index subgroup. But if $G/G'$ is finite then there is a bound for the order
of $F/F'$ as $F$ ranges over the finite images of $G$ because $F$ must have 
the form above, and so there is a bound for the order of $F$. Thus $R$ has
finite index in $G$ with $G/R$ equal to some $F$.
\end{proof}

\end{document}